\documentclass[reqno]{amsart}

\usepackage{amsmath,amscd}
\usepackage{amsfonts,amssymb}

\theoremstyle{plain}
\newtheorem{theorem}{Theorem}[section]

\newtheorem{lemma}[theorem]{Lemma}
\newtheorem{proposition}[theorem]{Proposition}
\theoremstyle{definition}

\theoremstyle{remark}
\newtheorem{remark}[theorem]{Remark}

\newcommand{\R}{{\mathbb R}}
\newcommand{\Q}{{\mathbb Q}}

\newcommand{\C}{{\mathbb C}}

\newcommand{\Z}{{\mathbb Z}}

\newcommand{\BC}{{\mathbb B}{\mathbb C}}
\newcommand{\BQ}{{\mathbb B}{\mathbb Q}}
\renewcommand{\bar}[1]{{\overline{#1}}}
\renewcommand{\i}{{\bf i}}
\renewcommand{\j}{{\bf j}}
\renewcommand{\k}{{\bf k}}
\newcommand{\e}{{\bf e}}
\newcommand{\edag}{{\bf e^\dagger}}

\newcommand{\calI}{{\mathcal I}}

\newcommand{\calO}{{\mathcal O}}

\newcommand{\calL}{{\mathcal L}}

\newcommand{\ip}[1]{\langle #1 \rangle}

\renewcommand{\Re}{{\operatorname{Re}}}

\numberwithin{equation}{section}

\begin{document}

\title[Zeta function for multicomplex algebra]{A zeta function for
  multicomplex algebra}

\author{A. Sebbar}%
\address{Institut de Math\'ematiques de Bordeaux, Universit\'e
  Bordeaux 1, 351 cours de la Lib\'eration, 33405 Talence cedex,
  France}%
\email{asebbar@math.u-bordeaux1.fr}%

\author{D.C. Struppa}%
\address{Chapman University, One University Drive, Orange CA 92866}%
\email{struppa@chapman.edu}%

\author{A. Vajiac}%
\address{Chapman University, One University Drive, Orange CA 92866}%
\email{avajiac@chapman.edu}%

\author{M. B. Vajiac}%
\address{Chapman University, One University Drive, Orange CA 92866}%
\email{mbvajiac@chapman.edu}%

\date{\today}%

\begin{abstract}
  In this paper we define and study a Dedekind-like zeta function for
  the algebra of multicomplex numbers. By using the idempotent
  representations for such numbers, we are able to identify this zeta
  function with the one associated to a product of copies of the field
  of Gaussian rationals. The approach we use is substantially
  different from the one previously introduced by Rochon (for the
  bicomplex case) and by Reid and Van Gorder (for the multicomplex
  case).
\end{abstract}

\maketitle


\section{Introduction}
\label{sec:intro}

In this paper we build on the resurgent interest for the theory of
bicomplex and multicomplex numbers, to develop a definition (and
discuss the fundamental properties) of a Dedekind-like zeta function
for the spaces of multicomplex numbers. Our approach is significantly
different from the one recently used by Rochon~\cite{rochon_zeta}, and
Reid and Van Gorder~\cite{reid}. We should note that zeta functions
play a significant role in a variety of fields, ranging from number
theory to statistical mechanics, from quantum field theory (where they
are used to regularize divergent series and divergent integrals) to
dynamical systems, and finally to the theory of crystals and
quasi-crystals (see e.g.~\cite{moody}). We believe that zeta functions
for multicomplex algebras will play an important role in a similar
range of applications. From a mathematical point of view, we observe
that the study of the case of multicomplex algebras represents only a
first step towards the understanding of the seminal work of
Hey~\cite{hey} and Artin~\cite{artin}, within the larger context of
quotient polynomial algebras. We plan to return to these issues in
future papers.

To begin with, and without pretense of completeness, we recall that
the space of bicomplex numbers arises when considering the space
$\mathbb{C}$ of complex numbers as a real bidimensional algebra, and
then complexifying it. With this process one obtains a four
dimensional algebra usually denoted by $\BC$. The key point of the
theory of functions on this algebra is that (despite the problems
posed by the existence of zero-divisors in $\BC$) the classical notion
of holomorphicity can be extended from one complex variable to this
algebra, and one can therefore develop a new theory of
(hyper)holomorphic functions. Modern references on this topic are
~\cite{alss}, and ~\cite{bicomplexbook}.

The algebra $\BC$ is therefore four dimensional over the reals, just
like the skew-field of quaternions, but while in the space of
quaternions we have three anti-commutative imaginary units, in the
case of bicomplex numbers one considers two imaginary units $\i,\j$
which commute, and so the third unit $\k=\i\j$ ends up being a ``new''
root of $1$; such units are usually called {\em hyperbolic}. Indeed,
every bicomplex number $Z$ can be written as $Z=z_1+\j z_2$, where
$z_1$ and $z_2$ are complex numbers of the form $z_1=x_1+\i y_1,$ and
$z_2=x_2+\i y_2$.  There are several ways to represent bicomplex
numbers, see ~\cite{bicomplexbook}, but the one that will be central
in this paper is called the idempotent representation of bicomplex
numbers, and will be described in
section~\ref{sec:multicomplex_algebras}.

One can then use a similar process to define the space $\BC_n$ of
multicomplex numbers, namely the space generated over the reals by $n$
commuting imaginary units. When $n=2$, the space of multicomplex
numbers is simply the space of bicomplex numbers. The history of
bicomplex numbers is not devoid of interest, and we refer the reader
to the recent ~\cite{cerroni}.

In~\cite{rochon_zeta}, the author introduced and studied the
properties of a holomorphic Riemann zeta function of two complex
variables in the context of the bicomplex algebra. Similarly,
in~\cite{reid} the authors defined a multicomplex Riemann zeta
function in the setup of multicomplex algebras. Both these studies
generalize the Riemann zeta function to several complex variables, in
the sense that in the definition of the original Riemann zeta
function,
$$
\zeta(s) := \sum_{n=1}^\infty \frac{1}{n^s}\,,
$$
the complex variable $s$ is replaced by a bicomplex, respectively a
multicomplex variable. Our approach, in this paper, is very
different. As it is well known, Dedekind generalized the Riemann zeta
function by considering an algebraic number field $K$, and defining
the associated Dedekind zeta function by
$$
\zeta_K(s):=\sum_{I \subset \mathcal{O}_K} \frac{1}{N(I)^s}
$$
where the sum ranges through all the non-zero ideals $I$ in the ring
of integers $\mathcal{O}_K$ of K, and (see
Section~\ref{sec:dedekind_zeta} for the full detail) $N(I)$ denotes
the norm of the ideal $I$. When $K=\mathbb{Q}$, the Dedekind zeta
function reduces to the Riemann zeta function.

Thus, it is natural to look at the Dedekind approach for quadratic
fields, and concurrently the Hey~\cite{hey} approach for hypercomplex
algebras, as a way to define a Dedekind-like zeta function in the
context of the bicomplex and multicomplex vector spaces $\BQ$,
respectively $\BQ_n$.  As the reader will see, the crucial point in
being able to explicitly calculate this type of zeta function
(sometimes called the Hey zeta function in the literature) for
multicomplex numbers is the existence of the idempotent
representations of bicomplex and multicomplex numbers. This
representation will allow us to identify the Hey zeta functions for
products of copies of $\Q(\i)$, the field of Gaussian rationals, and
as a result we will have an explicit formula for such a function.


\bigskip

The architecture of the paper is as
follows. Section~\ref{sec:quadratic_fields} gives a self-contained
review of quadratic fields, so that all necessary results are
available to the reader. In particular, we will define the quadratic
$L$-function of a quadratic field $K$, and we will calculate the value
of its analytic extension for $s=1.$ Section~\ref{sec:dedekind_zeta}
provides all the necessary background on the Dedekind's zeta
function. In Section~\ref{sec:multicomplex_algebras} we give all the
necessary information on bicomplex and multicomplex algebras. The core
of the paper is Section~\ref{sec:bicomplex_zeta} where we utilize the
instruments introduced previously to calculate explicitly the
Dedekind-like (Hey) zeta function for the algebra $\BQ$ of bicomplex
numbers with rational coefficients. We finally show how this results
extends to the multicomplex case, and we explicitly write the
functional equation satisfied by these functions.

{\bf Acknowledgments}. We would like to thank Aurel Page for fruitful
discussions about the topic of this paper.

\bigskip

\section{Review on Quadratic Fields}
\label{sec:quadratic_fields}

For the convenience of the reader and in order to make the paper
self-contained, we summarize the main definitions and results for
quadratic fields. Notations, definitions, and results here
follow~\cite{ireland,neukirch} and the references therein.  

A {\em number field} is a finite degree field extension $K$ over
$\Q$. We denote by $\calO_K$ the ring of integers of $K$, i.e. the
ring of elements $\alpha\in K$ that are roots of monic polynomials in
$\Z[X]$.

A {\em quadratic field} is a degree two extension of $\Q$. It has the
form $K=\Q(\sqrt{d})$, where $d$ is a square free integer different
than 1. The ring of integers in this case is
\begin{align*}
  \calO_d := \calO_{\Q(\sqrt{d})} = \left\{
  \begin{array}{cc}
    \Z + \Z\sqrt{d}, & \text{if } d\not\equiv 1\mod 4\\
    \,\\
    \Z + \Z\,\displaystyle\frac{1+\sqrt{d}}{2}, & \text{if } d\equiv 1\mod 4
  \end{array}\right.
\end{align*}
The {\em norm} of an element $\alpha=a+b\sqrt d\in\Q(\sqrt d)$ is the
(non-necessarily positive) integer defined by
\begin{align*}
  N(\alpha) := (a+b\sqrt d)(a-b\sqrt d) = a^2 - db^2\,.
\end{align*}
The norm of an ideal $I$ of $\calO_K$ is
defined by
\begin{align*}
  N(I) := |\calO_{K}/I|\,,
\end{align*}
where we note that the quotient ring $\calO_{K}/I$ is always of finite
cardinality for each number field $K$. A particular case occurs when
$I=(\alpha)$ is a principal ideal, where $\alpha=a+b\sqrt d$. Then
\begin{align*}
  N(I) = N((\alpha)) = |a^2 - db^2|\,.
\end{align*}

For an odd prime $p\in\Z$, the Legendre symbol is defined by:
\begin{align*}
  \left(\frac{a}{p}\right) :=\left\{
  \begin{array}{ll}
    1, & \text{if } a^{\frac{p-1}{2}} \equiv 1\mod p\\
    -1, & \text{if } a^{\frac{p-1}{2}} \equiv -1\mod p\\
    0, & \text{if } a \equiv 0\mod p\,.
  \end{array}\right.
\end{align*}
Recall that to say that $a$ is a quadratic residue modulo $p$ means
that the equation $x^2=a\mod p$ has a solution. We can therefore
reformulate the definition of Legendre symbol as follows:
\begin{align*}
  \left(\frac{a}{p}\right) :=\left\{
  \begin{array}{ll}
    1, & \text{if } a \text{ is a quadratic residue modulo } p \text{ and } a\not\equiv 0\mod p\\
    -1, & \text{if } a \text{ is not a quadratic residue modulo } p \\
    0, & \text{if } a\equiv 0\mod p \,.
  \end{array}\right.
\end{align*}

An extension of the Legendre symbol, due to Kronecker, is the
following. Each integer $n$ has a prime factorization
\begin{align*}
  n = u p_1^{\ell_1}\cdots p_k^{\ell_k}\,,
\end{align*}
where $u=\pm 1$ and each $p_i, 1\leq i\leq k$ is prime. The
Kronecker-Legendre symbol is defined by:
\begin{align*}
  \left(\frac{a}{n}\right) := \left(\frac{a}{u}\right)
  \prod_{i=1}^k \left(\frac{a}{p_i}\right)^{\ell_i}\,,
\end{align*}
where:
\begin{enumerate}
\item for odd prime $p$, $\displaystyle\left(\frac{a}{p}\right)$
  is the Legendre symbol;
\item for $p=2$, we define
  \begin{align*}
      \left(\frac{a}{p}\right) :=\left\{
        \begin{array}{ll}
          0, & \text{if } a \text{ is even}\\
          1, & \text{if } a \equiv\pm 1\mod 8\\
          -1, & \text{if } a \equiv \pm 3\mod 8 \,.
        \end{array}\right.
    \end{align*}
\item for $u=\pm 1$ we define
  \begin{align*}
    \left(\frac{a}{1}\right)=1,\qquad
    \left(\frac{a}{-1}\right)=\left\{
        \begin{array}{ll}
          1, & \text{if } a \geq 0\\
          -1, & \text{if } a < 0\,.
        \end{array}\right.
  \end{align*}
\item we define
  \begin{align*}
    \left(\frac{a}{0}\right)=\left\{
        \begin{array}{ll}
          1, & \text{if } a =\pm 1\\
          0, & \text{otherwise}\,.
        \end{array}\right.
  \end{align*}
\end{enumerate}

A fundamental result for our study of the Dedekind zeta function is
the following
\begin{theorem}
  Every non-zero ideal of $\calO_d$ can be written as a product of
  prime ideals. The decomposition is unique up to the order of the
  factors.
\end{theorem}

The {\em discriminant} of the quadratic number field $K=\Q(\sqrt d)$
is defined by
\begin{align*}
  \Delta = \Delta_K = \left\{
  \begin{array}{ll}
    d & \text{if } d\equiv 1\mod 4\\
    4d & \text{if } d\equiv 2,3\mod 4\,.
  \end{array}\right.
\end{align*}
Some primes in $\Z$ are not prime elements in $\calO_d$\,: for example,
in the case $d=-1$, we have:
\begin{align*}
  2 &= \i(1-\i)^2,\\
  5 &= (2+\i)(2-\i)\,.
\end{align*}
This situation is described precisely by the Legendre symbol. If $p$
is a rational prime (i.e. prime in $\Z$), then the ideal
$(p) = p\calO_d$ of $\calO_d$ has the following form: 
\begin{align*}
  (p) = \left\{
  \begin{array}{ll}
    \mathfrak{p} \mathfrak{p}' (\text{where }\mathfrak{p}\neq\mathfrak{p}'), &
                                 \text{ if } \displaystyle\left(\frac{\Delta}{p}\right)=1\\
    \,\\
    \mathfrak{p},  & \text{ if } \displaystyle\left(\frac{\Delta}{p}\right)=-1\\
    \,\\
    \mathfrak{p}^2,  & \text{ if } \displaystyle\left(\frac{\Delta}{p}\right)=0,
  \end{array}\right.
\end{align*}
where $\mathfrak{p}, \mathfrak{p}'$ are prime ideals of $\calO_d$. We
respectively say that in these cases the ideal $(p)$ splits, stays
inert, or ramifies in $\calO_d$.

A subset $F$ of $\calO_d$ is called a {\em fractional ideal} of
$\calO_d$ if there exists $\beta\in\calO_d$, $\beta\neq 0$, such that
$\beta F$ is an ideal of $\calO_d$. Then we have, for some
ideal $I$ of $\calO_d$,
\begin{align*}
  F = \left\{ \frac{\alpha}{\beta}\,\bigg|\, \alpha\in I\right\}\,,
\end{align*}
The set $F_{\Q(\sqrt d)}$ of fractional ideals of $\calO_d$ can be
equipped with an abelian group structure, as follows. If $I_1,I_2$ are ideals
of $\calO_d$, and
\begin{align*}
  F_1 = \left\{ \frac{\alpha_1}{\beta_1}\,\bigg|\, \alpha_1\in I_1\right\}\quad
  F_2 = \left\{ \frac{\alpha_2}{\beta_2}\,\bigg|\, \alpha_2\in I_2\right\}\,,
\end{align*}
where $\beta_1,\beta_2\in\calO_d$, we define:
\begin{align*}
  F_1F_2 := \left\{ \frac{\alpha}{\beta_1\beta_2}\,\bigg|\, \alpha\in I_1I_2\right\}
\end{align*}
where $I_1I_2$ is the ideal generated by all products
$\alpha_1\alpha_2$, with $\alpha_1\in I_1$, $\alpha_2\in I_2$.

The identity element
is $\calO_d$ and the inverse of a fractional ideal $F$ is given by
\begin{align*}
  F^{-1} = \left\{ \alpha\in\Q(\sqrt d)\,\bigg|\, \alpha F\subset \calO_d\right\}\,.
\end{align*}
The set of principal fractional ideals
$B_{\Q(\sqrt d)}\subset F_{\Q(\sqrt d)}$ is a subgroup of
$F_{\Q(\sqrt d)}$, and the quotient
\begin{align*}
  H_d := F_{\Q(\sqrt d)}\big/ B_{\Q(\sqrt d)}
\end{align*}
is called the {\em ideal class group} of $\Q(\sqrt d)$. Its order
$h_d$ is the {\em class number} of $\Q(\sqrt d)$. This notion measures
how far the ring $\calO_d$ is from being principal.

More precisely, if $h_d=1$ then there is only one
equivalence class in $H_d$, and each fractional ideal is equivalent to
the principal ideal $(1)=\calO_d$ modulo multiplication by principal
ideals. That is, for each fractional ideal $A$ there exists
$\alpha\in\calO_d$ such that
\begin{align*}
  (\alpha)=\alpha A = (1) = \calO_d\,,
\end{align*}
then
\begin{align*}
  A = \left( \frac{1}{\alpha}\right)\,.
\end{align*}
Hence each fractional ideal and, therefore, each ideal is principal.

There exist nine imaginary quadratic fields of class number one. They
are $\Q(\sqrt d)$ for
\begin{align*}
  d = -1, -2, -3, -7, -11, -19, -43, -67, -163.
\end{align*}

If $K=\Q(\sqrt d)$, $d<0$, all norms $a^2-db^2$ are non-negative and
the unit group in $\calO_d$ is
\begin{align*}
  \calO_d^\times = \left\{
  \begin{array}{ll}
    \{\pm 1,\pm \i\} & \text{if } d=-1\\
    \{\pm 1,\pm \zeta_3,\pm\zeta_3^2\} & \text{if } d=-3\\
    \{\pm 1\} & \text{otherwise }\,,
  \end{array}\right.
\end{align*}
where $\zeta_3$ is the principal cubic root of unity. Therefore, the
order of the group $\calO_d^\times$ of units is:
\begin{align*}
  w_d = w = 4, 6, 2
\end{align*}
according to the values $d=-1$, $d=-3$, or $d\neq -1,-3$,
respectively.

For real quadratic fields $\Q(\sqrt d)$, $d>0$, the situation is very
different. In this case, the group of units is infinite and has the
form
\begin{align*}
  \calO_d^\times = \left\{ \pm u_d^n\,\big|\, n\in\Z\right\}\simeq \Z/2\Z\times\Z\,,
\end{align*}
where $u_d>1$ is the so-called {\em fundamental unit}. It is a
difficult problem to find $u_d$.

For real quadratic fields $K=\Q(\sqrt d)$, $d>0$, the logarithm of the
fundamental unit is called the {\em regulator}. For imaginary quadratic
fields, the regulator is 1, as it is for the ring of (rational)
integers $\Z$. For example, if $d=2$, $u_2 = 1+\sqrt 2$ and the
regulator $R_2$ of $\Q(\sqrt 2)$ is
\begin{align*}
  R_2 = \log(1+\sqrt 2)\,,
\end{align*}
where $1+\sqrt 2$ is the fundamental solution of the Pell equation
\begin{align*}
  x^2 - 2y^2 =1.
\end{align*}

Let $K$ be a quadratic field with discriminant $\Delta$, so that
$\Delta = d$, if $d\equiv 1\mod 4$ or $4d$ if $d\equiv 2,3\mod 4$. The
quadratic character of $K$ is the morphism
\begin{align*}
  \chi_K:(\Z,+)\to\C,\qquad \chi_K(m):=\left(\frac{\Delta}{m}\right)\,.
\end{align*}
It is a fundamental property that $\chi_K$ is periodic with period $|\Delta|$.

For $K=\Q(\sqrt d)$, we can also define the Dirichlet character, also
denoted by $\chi_K$, by
\begin{align*}
  \chi_K&:\left(\Z\big/|\Delta|\Z\right)^\times\to\C^\times,\\
  \chi_K\left(m+|\Delta|\Z\right)&=\left\{
    \begin{array}{ll}
     \left(\frac{m}{d}\right), & \text{if } d\equiv 1\mod 4\\
      \,\\
     (-1)^{\frac{m-1}{2}}\left(\frac{m}{d}\right), & \text{if } d\equiv 3\mod 4\\
      \,\\
  (-1)^{\frac{m^2-1}{8}}\left(\frac{m}{\left(\frac{d}{2}\right)}\right), & \text{if } d\equiv 2\mod 8\\
      \,\\
  (-1)^{\frac{(m-1)(m+5)}{8}}\left(\frac{m}{\left(\frac{d}{2}\right)}\right), & \text{if } d\equiv 6\mod 8\\
    \end{array}\right.
\end{align*}

In the case $K=\Q(\i)$, $\calO_{\Q(\i)} = \Z[\i]$, the ring of
Gaussian integers, and the Dirichlet character is
\begin{align*}
  \chi_{\Q(\i)}:\left(\Z\big/4\Z\right)^\times\to\C^\times,\qquad
  \chi_{\Q(\i)}(m) = (-1)^{\frac{m-1}{2}}\,.
\end{align*}


We can formulate the decomposition of rational primes as follows. Let
$p$ be a rational prime. The decomposition of $(p)$ in $\calO_K$ is
given by:
\begin{align*}
  p\calO_K = \left\{
  \begin{array}{ll}
    \mathfrak{p} \mathfrak{p}', \text{where }N(\mathfrak{p})=N(\mathfrak{p}')=p, &
                                 \text{ if } \chi_K(p)=1\\
    \,\\
    (p), \text{where }N((p))=p^2,  & \text{ if } \chi_K(p)=-1\\
    \,\\
    (p)^2, \text{where }N((p))=p,  & \text{ if } \chi_K(p)=0
  \end{array}\right.
\end{align*}

The quadratic $L$-function of a quadratic field $K$ with discriminant
$\Delta$ is given by:
\begin{align*}
  L(\chi_K,s) = \sum_{n\geq 1} \chi_K(n) n^{-s}
  = \prod_{p\text{ prime}} \left(1-\chi_K(p)p^{-s}\right)^{-1}\,.
\end{align*}
The sum is defined for $\Re(s)> 1$, but actually the quadratic
$L$-function extends analytically to $\Re(s)>0$. The value of
$L(\chi_K,1)$ is a remarkable one (see
e.g.~\cite{cohen,ireland,lang}).
\begin{proposition}
  (i) For a real quadratic field $K$:
  \begin{align*}
    L(\chi_K,1) = -\frac{1}{\sqrt{|\Delta|}} \sum_{r=1}^{|\Delta|-1}
    \chi_K(r)\log\left(\sin\left(\frac{\pi r}{|\Delta|}\right)\right)\,.
  \end{align*}
  (ii) For an imaginary quadratic field $K$:
  \begin{align*}
    L(\chi_K,1) = -\frac{\pi}{|\Delta|^{\frac32}} \sum_{r=1}^{|\Delta|-1}
    \chi_K(r) r\,.
  \end{align*}
\end{proposition}

\bigskip

\section{The Dedekind zeta function}
\label{sec:dedekind_zeta}

The Dedekind zeta function of a quadratic field $K$ is given by:
\begin{align}
  \label{def:dedekind}
  \zeta_K(s) = \sum_{\mathfrak{a}}N(\mathfrak{a})^{-s}
  = \prod_{\mathfrak{p}}\left(1-N(\mathfrak{p})^{-s}\right)^{-1}\,.
\end{align}
The sum is over all non-zero ideals $\mathfrak{a}$ of $\calO_K$ and
the product is over all prime ideals of $\calO_K$. If denote by
\begin{align*}
  a_n = \#\{\mathfrak{a}\,\big|\, N(\mathfrak{a}) = n \}\,,
\end{align*}
then
\begin{align*}
  \zeta_K(s) = \sum_{n>0} \frac{a_n}{n^s}\,.
\end{align*}

The fundamental property of the Dedekind zeta function is the
following factorization:
\begin{theorem}
  For $\Re(s)>1$, we have
  \begin{align*}
    \zeta_K(s) = \zeta(s) L(\chi_K,s)\,,
  \end{align*}
  where $s\mapsto \zeta(s)$ is the Riemann zeta function.
\end{theorem}
We give below an idea of the proof. For $\Re(s)>1$ we can write:
\begin{align}\label{zeta decomposition}
  \zeta_K(s) = \prod_{p} \prod_{\mathfrak{p}|(p)} \left(1-N(\mathfrak{p})^{-s}\right)^{-1}\,,
\end{align}
where $(p)=p\,\calO_K$ is the ideal generated by $p$ in
$\calO_K$. Furthermore, for $\Re(s)>1$ we have:
\begin{align*}
  \zeta(s) L(\chi_K,s) = \prod_{p} (1-p^{-s})^{-1} (1-\chi(p) p^{-s})^{-1}
\end{align*}
\begin{lemma}
  \begin{align*}
    \prod_{\mathfrak{p}|(p)} \left(1-N(\mathfrak{p})^{-s}\right)
    = (1-p^{-s})(1-\chi_K(p) p^{-s})\,.
  \end{align*}
\end{lemma}
\begin{proof}
  Indeed, for a given rational prime $p$, if $\chi_K(p)=1$, then
  \begin{align*}
    (p) = p\,\calO_K = \mathfrak{p}\mathfrak{p}',\qquad
    N(\mathfrak{p})=N(\mathfrak{p}')=p\,,
  \end{align*}
  and both sides are $(1-p^{-s})^2$.  If $\chi_K(p)=-1$, then $p$ is
  inert and $\mathfrak{p}=(p)$, with $N(\mathfrak{p})=p^2$, and both
  sides are
  \begin{align*}
    (1-p^{2(-s)}) = (1-p^{-s})(1+p^{-s})\,.
  \end{align*}
  Finally, if $\chi_K(p)=0$, then both sides are $1-p^{-s}$.
\end{proof}
From this lemma, we obtain
\begin{align*}
  \zeta_K(s) &= \prod_p \prod_{\mathfrak{p}|(p)} \left(1-N(\mathfrak{p})^{-s}\right)^{-1}\\
  &= \prod_p (1-p^{-s})^{-1} (1-\chi_K(p) p^{-s})^{-1}\\
  &= \zeta(s) L(\chi_K,s)\,.
\end{align*}

A beautiful consequence of what we have seen is the Dirichlet class
number formula for imaginary quadratic fields. It says that:
\begin{align*}
  \frac{2\pi h}{w \sqrt{|\Delta|}} = L(\chi_K,1)\,,
\end{align*}
where $h$ is the ideal class number of $K$, and $w$ is the number of
roots of unity in $K$.

\bigskip

We turn now to the case $K=\Q(\i)$, which is our main interest. The
Dedekind zeta function of $\Q(\i)$ is given by:
\begin{align}
  \label{eq:dedekind_zeta}
  \zeta_{\Q(\i)} &= \zeta(s) L(\chi_{-4},s)\\
                 &= \frac{1}{1-2^{-s}} \prod_{p\equiv 1(\text{mod } 4)}
                   \frac{1}{(1-p^{-s})^2}\prod_{p\equiv 3(\text{mod } 4)}
                   \frac{1}{1-p^{-2s}},\nonumber
\end{align}
where, in order to simplify the notations, we write $\chi_{-4}$
instead of $\chi_{\Q(\i)}$, since $-4$ is the discriminant of the
field $\Q(\i)$.  The Dirichlet $L$-series is then:
\begin{align*}
  L(\chi_{-4},s) = \sum_{n\geq 1} \frac{\chi_{-4}(n)}{n^s}\,.
\end{align*}
Note that $\chi_{-4}$ is the character of $\Z$, of period $4$, given
by:
\begin{align*}
  \chi_{-4}(n) = \left\{
  \begin{array}{ll}
    0, & \text{if } n\equiv 0\mod 4\\
    1, & \text{if } n\equiv 1\mod 4\\
    0, & \text{if } n\equiv 2\mod 4\\
    -1, & \text{if } n\equiv 3\mod 4
  \end{array}\right.\,.
\end{align*}
This implies that
\begin{align*}
  L(\chi_{-4},s) = \sum_{n\geq 0} \frac{1}{(4n+1)^s} - \sum_{n\geq 0} \frac{1}{(4n+3)^s}\,.
\end{align*}

\bigskip

To study the analytic continuation of $L(\chi_{-4},s)$ we relate it to
the Hurwitz zeta function, defined by:
\begin{align*}
  \zeta(s,\alpha) = \sum_{n\geq 0} \frac{1}{(n+\alpha)^s},\qquad
  \Re(s)>1, \quad\alpha>0.
\end{align*}
This function reduces to the Riemann zeta function for $\alpha=1$. The
general theory says that $s\mapsto \zeta(s,\alpha)$ admits a
meromorphic continuation to the whole plane, with a simple pole at
$s=1$ with residue 1. We need the special value
\begin{align*}
  \lim_{s\to 1}\left(\zeta(s,\alpha) - \frac{1}{s-1}\right)
  = -\frac{\Gamma'(\alpha)}{\Gamma(\alpha)}
\end{align*}
where $\Gamma$ is the Euler function (see e.g.~\cite{bateman}). An
immediate consequence of this fact is that $s\mapsto L(\chi_{-4},s)$
is analytically continuable to the whole plane, according to the
equality:
\begin{align*}
  L(\chi_{-4},s) = 4^{-s}\left(\zeta\left(s,\frac14\right) - \zeta\left(s,\frac34\right)\right)\,.
\end{align*}

Another representation of $L(\chi_{-4},s)$ is
\begin{align*}
  L(\chi_{-4},s) = \sum_{n=0}^{\infty} \frac{(-1)^n}{(2n+1)^s} = \beta(s)\,,
\end{align*}
where $\beta$ is the Dirichlet Beta function.

The function $s\mapsto\zeta(s)L(\chi_{-4},s)$ extends meromorphically
to all $\C$, with a simple pole at $s=1$. We then have, for
$0<|s-1|<\infty$:
\begin{align*}
  \zeta(s)L(\chi_{-4},s) = \frac{C_{-1}}{s-1} + C_0 + C_1(s-1) + C_2(s-1)^2+\cdots
\end{align*}
We give below the values of $C_{-1}$ and $C_0$. From the expansions:
\begin{align*}
  \zeta(s) = \frac{1}{s-1} + \gamma + \gamma_1(s-1)+\cdots
\end{align*}
where $\gamma$ is the (small) Euler constant, and
\begin{align*}
  L(\chi_{-4},s) = L(\chi_{-4},1) + L'(\chi_{-4},1)(s-1) + \cdots
\end{align*}
we obtain
\begin{align*}
  \zeta(s)L(\chi_{-4},s) = \frac{L(\chi_{-4},1)}{s-1} + L'(\chi_{-4},1) + \gamma L(\chi_{-4},1) + \cdots
\end{align*}
Hence $L(\chi_{-4},1)$ is the residue of $\zeta(s)L(\chi_{-4},s)$ at $s=1$ and
\begin{align*}
  \gamma_{\Q(\i)} := L'(\chi_{-4},1) + \gamma L(\chi_{-4},1)
\end{align*}
is what we may call the Euler constant $\gamma_{\Q(\i)}$ of the field
$\Q(\i)$ (as $\gamma=\gamma_\Q$).
\begin{align*}
  L(\chi_{-4},1) &= \lim_{s\to 1} 4^{-s}\left(\zeta\left(s,\frac14\right) - \zeta\left(s,\frac34\right)\right)\\
                &= \frac14\left(\lim_{s\to 1}\left(\zeta\left(s,\frac14\right)-\frac{1}{s-1}\right)-
                  \lim_{s\to 1}\left(\zeta\left(s,\frac34\right)-\frac{1}{s-1}\right)\right)\\
                &= \frac14\left(\frac{\Gamma'\left(\frac34\right)}{\Gamma\left(\frac34\right)}
                  -\frac{\Gamma'\left(\frac14\right)}{\Gamma\left(\frac14\right)}\right)\,.
\end{align*}
The logarithmic derivative $\displaystyle\frac{\Gamma'}{\Gamma}$ of
the $\Gamma$-function is a remarkable function. We only need to know
it satisfies the functional equation
\begin{align*}
  \frac{\Gamma'(z)}{\Gamma(z)} - \frac{\Gamma'(1-z)}{\Gamma(1-z)}
  = \pi\cot(\pi z)\,.
\end{align*}
Hence
\begin{align*}
  L(\chi_{-4},1) = -\frac14\left(\frac{\Gamma'\left(\frac14\right)}{\Gamma\left(\frac14\right)}
                  -\frac{\Gamma'\left(\frac34\right)}{\Gamma\left(\frac34\right)}\right)
  =\frac{\pi}{4}\,.
\end{align*}

To find the Euler constant $\gamma_{\Q(\i)}$ of the field $\Q(\i)$, we
observe that:
\begin{align*}
  \frac{\gamma_{\Q(\i)}}{L(\chi_{-4},1)} = \gamma + \frac{L'(\chi_{-4},1)}{L(\chi_{-4},1)}\,,
\end{align*}
which is what it is called the Sierpinski constant (see
e.g.~\cite{finch}):
\begin{align*}
  \gamma + \frac{L'(\chi_{-4},1)}{L(\chi_{-4},1)} &=
        \log\left(2\pi e^{2\gamma} \frac{\Gamma\left(\frac34\right)^2}{\Gamma\left(\frac14\right)^2}\right)\\
        &=\frac{\pi}{3} - \log 4 + 2\gamma - 4\sum_{k=1}^\infty\log\left(1-e^{-2\pi k}\right)\\
        &= 0.8228252\dots
\end{align*}
Hence the analogue of the Euler constant for $\Q(\i)$ is:
\begin{align*}
  \gamma_{\Q(\i)} &= L'(\chi_{-4},1) + \gamma L(\chi_{-4},1)\\
  &= \frac{\pi}{4}\log\left(2\pi e^{2\gamma}
    \frac{\Gamma\left(\frac34\right)^2}{\Gamma\left(\frac14\right)^2}\right)\\
  &= \frac{\pi}{2}\left(\gamma + \log 2 + \frac32\log\pi -2\log\Gamma\left(\frac14\right)\right)\,,
\end{align*}
where we used the classical complement formula for the
$\Gamma$-function:
\begin{align*}
  \Gamma(z)\Gamma(1-z) = \frac{\pi}{\sin(\pi z)},
  \qquad z\neq 0,\pm 1,\pm 2,\dots
\end{align*}

\bigskip

\begin{remark}
  \label{remark}
  We would like to give an idea on how to calculate the coefficients
  of the Dirichlet series of $\zeta_{\Q(\i)}(s)$. First, for
  $\Re(s)>1$, we have:
  \begin{align*}
    \zeta_{\Q(\i)}(s) = \frac14\sum_{(m,n)\neq(0,0)} \frac{1}{(m^2+n^2)^s}
    = \frac14\sum_{n=1}^\infty \frac{r_2(n)}{n^s}\,,
  \end{align*}
  where
  \begin{align*}
    r_2(n) = \big|\{(p,q)\in\Z\times\Z,\quad p^2+q^2=n\}\big|
  \end{align*}
  is the number of representation of $n$ as a sum of two squares.

  Secondly, we have:
  \begin{align*}
    \zeta_{\Q(\i)}(s) &= \zeta(s) L(\chi_{-4},s) 
                        = \left(\sum_{n=1}^\infty \frac{1}{n^s}\right) 
                        \left(\sum_{n=1}^\infty \frac{\chi_{-4}(n)}{n^s}\right)\\
                      &= \sum_{n_1,n_2=1}^\infty \frac{\chi_{-4}(n_2)}{n_1^sn_2^s}
                        = \sum_{n=1}^\infty \frac{\left(\displaystyle\sum_{d | n} \chi_{-4}(d)\right)}{n^s}\,,
  \end{align*}
  so that 
  \begin{align*}
    \zeta_{\Q(\i)}(s) = \sum_{n=1}^\infty \frac{a_n}{n^s}\,,
  \end{align*}
  where
  \begin{align*}
    a_n = \sum_{d | n} \chi_{-4}(d) = \frac14 r_2(n)
    = \frac14 \big|\{(p,q)\in\Z\times\Z,\quad p^2+q^2=n\}\big|\,.
  \end{align*}
  For $\Re(s)>1$ we obtain:
  \begin{align*}
    \zeta_{\Q(\i)}(s) = 1 + \frac{1}{2^s} + \frac{1}{4^s} + \frac{2}{5^s} 
    + \frac{1}{8^s} + \frac{2}{10^s} + \frac{2}{13^s} 
    + \frac{1}{16^s} + \frac{2}{17^s} + \frac{1}{18^s} + \frac{2}{20^s} 
    + \frac{3}{25^s} + \cdots
  \end{align*}
  A nice application of the formula for $r_2(n)$ given above is the
  following (see~\cite{rademacher}). Let $\sigma_0^{(1)}(n)$ and
  $\sigma_0^{(3)}(n)$ be the number of divisors of $n$ congruent to 1
  and 3 modulo 4, respectively. Then we have:
  \begin{align*}
    \sum_{d | n} \chi_{-4}(d) =  \sigma_0^{(1)}(n) - \sigma_0^{(3)}(n)\,,
  \end{align*}
  and so
  \begin{align*}
    r_2(n) =  \left(\sigma_0^{(1)}(n) - \sigma_0^{(3)}(n)\right).
  \end{align*}
  If follows that for each positive integer $n$, we have
  $\sigma_0^{(1)}(n) \geq \sigma_0^{(3)}(n)$ (it means that there are
  more divisors of $n$ congruent to 1 (mod 4) than congruent to 3 (mod
  4)). Furthermore, if $p$ is prime, $p\equiv 1\mod 4$, then $r_2(p)=8$,
  which is a famous theorem of Fermat (see e.g.~\cite{rademacher}).
\end{remark}

\bigskip

\section{Multicomplex Algebras}
\label{sec:multicomplex_algebras}

Without giving many details (for which we refer the reader to the
fairly comprehensive recent
references~\cite{bicomplex2,bicomplexbook,price,bicomplex3,bicomplex4}),
we will simply say that the space $\BC_n$ of multicomplex numbers is
the space generated over the reals by $n$ commuting imaginary
units. The algebraic properties of this space and analytic properties
of multicomplex valued functions defined on $\BC_n$ has been studied
in~\cite{bicomplex4}.

In the case of only one imaginary unit, denoted by $\i_1$, the space
$\BC_1$ is the usual complex plane $\C$. Since, in what follows, we
will have to work with different complex planes, generated by
different imaginary units, we will denote such a space also by
$\C(\i_1)$, in order to clarify which imaginary unit is used in the
space itself.

The next case occurs when we have two commuting imaginary units $\i_1$
and $\i_2$. This yields the bicomplex space $\BC_2$, or $\BC$.  For
simplicity of notation, we will relabel the units as
\begin{align*}
  \i := \i_1,\qquad \j:= \i_2,\qquad \k := \i\j = \i_1\i_2\,.
\end{align*}
Note that $\k$ is a hyperbolic unit, i.e. it is a unit which squares
to $1$.  Because of these various units in $\BC$, there are several
different conjugations that can be defined naturally. We will not make
use of these conjugations in this paper, but we refer the reader
to~\cite{bicomplexbook}.

\medskip

$\BC$ is not a division algebra, and it has two distinguished zero
divisors, $\e$ and $\edag$, which are idempotent, linearly independent
over the reals, and mutually annihilating with respect to the
bicomplex multiplication:
\begin{align*}
  \e&:=\frac{1+\k}{2}\,,\qquad \edag:=\frac{1-\k}{2}\,,\\
  \e \cdot \edag &= 0,\qquad
  \e^2=\e , \qquad (\edag)^2 =\edag\,,\\
  \e +\edag &=1, \qquad \e -\edag = \k\,.
\end{align*}
Just like $\{1, \j \},$ they form a basis of the complex algebra
$\BC$, which is called the {\em idempotent basis}. If we define the
following complex variables in $\C(\i)$:
\begin{align*}
  \beta_1 := z_1-\i_1z_2,\qquad \beta_2 := z_1+\i_1z_2\,,
\end{align*}
the $\C(\i)$--{\em idempotent representation} for $Z=z_1+\i_2 z_2$ is
given by
\begin{align*}
  Z &= \beta_1\e+\beta_2\edag\,.
\end{align*}

The $\C(\i)$--idempotent is the only representation for which
multiplication is component-wise, as shown in the next proposition.

\begin{proposition}
  \label{prop:idempotent}
  The addition and multiplication of bicomplex numbers can be realized
  component-wise in the idempotent representation above. Specifically,
  if $Z= a_1\,\e + a_2\,\edag$ and $W= b_1\,\e + b_2\,\edag $ are two
  bicomplex numbers, where $a_1,a_2,b_1,b_2\in\C(\i)$, then
  \begin{eqnarray*}
    Z+W &=& (a_1+b_1)\,\e  + (a_2+b_2)\,\edag   ,  \\
    Z\cdot W &=& (a_1b_1)\,\e  + (a_2b_2)\,\edag   ,  \\
    Z^n &=& a_1^n \,\e  + a_2^n \,\edag  .
  \end{eqnarray*}
  Moreover, the inverse of an invertible bicomplex number
  $Z=a_1\e + a_2\edag $ (in this case $a_1 \cdot a_2 \neq 0$) is given
  by
  $$
  Z^{-1}= a_1^{-1}\e + a_2^{-1}\,\edag ,
  $$
  where $a_1^{-1}$ and $a_2^{-1}$ are the complex multiplicative
  inverses of $a_1$ and $a_2$, respectively.
\end{proposition}

One can see this also by computing directly which product on the
bicomplex numbers of the form
\begin{align*}
  x_1 + \i x_2 + \j x_3 + \k x_4,\qquad x_1,x_2,x_3,x_4\in\R
\end{align*}
is component wise, and one finds that the only one with this property
is given by the mapping:
\begin{align}
  \label{shakira}
  x_1 + \i x_2 + \j x_3 + \k x_4 \mapsto ((x_1 + x_4) + \i (x_2-x_3), (x_1-x_4) + \i (x_2+x_3))\,,
\end{align}
which corresponds exactly with the idempotent decomposition
\begin{align*}
  Z = z_1 + \j z_2 = (z_1-\i z_2)\e + (z_1+\i z_2)\edag\,,
\end{align*}
where $z_1 = x_1+\i x_2$ and $z_2 = x_3+\i x_4$.

\bigskip

The principal ideals $\ip{\e}$ and $\ip{\edag}$ generated by $\e$ and
$\edag$ in $\BC$ have the following properties:
\begin{align*}
  \ip{\e}\cdot \ip{\edag} = \{0\},\qquad
  \ip{\e} \cap \ip{\edag} = \{0\},\qquad
  \ip{\e} + \ip{\edag} = \BC\,,
\end{align*}
so they are coprime ideals in $\BC$.

\bigskip

We now turn to the definition of the multicomplex spaces, $\BC_n,$ for
values of $n\geq 2$. These spaces are defined by taking $n$ commuting
imaginary units $\i_1, \i_2,\dots,\i_n$ i.e. $\i_a^2=-1,$ and
$\i_a\i_b=\i_b\i_a$ for all $a,b=1,\dots,n.$ Since the product of two
commuting imaginary units is a hyperbolic unit, and since the product
of an imaginary unit and a hyperbolic unit is an imaginary unit, we
see that these units will generate a set $\mathfrak{A}_n$ of $2^n$
units, $2^{n-1}$ of which are imaginary and $2^{n-1}$ of which are
hyperbolic units. Then the algebra generated over the real numbers by
$\mathfrak{A}_n$ is the multicomplex space $\BC_n$ which forms a ring
under the usual addition and multiplication operations. As in the case
$n=2$, the ring $\BC_n$ can be represented as a real algebra, so that
each of its elements can be written as
$Z=\sum_{I\in \mathfrak{A}_n} Z_I I$, where $Z_I$ are real numbers.

In particular, following~\cite{price}, it is natural to define the
$n$-dimensional multicomplex space as follows:
$$
\BC_n:=\{Z_n=Z_{n-1,1}+\i_nZ_{n-1,2}\,\big|\,Z_{n-1,1},Z_{n-1,2}\in\BC_{n-1}\}
$$
with the natural operations of addition and multiplication. Since
$\BC_{n-1}$ can be defined in a similar way using the $\i_{n-1}$ unit
and multicomplex elements of $\BC_{n-2}$, we recursively obtain, at
the $k-$th level:
$$
Z_n=\sum_{|I|=n-k}\,\prod_{t=k+1}^{n} (\i_t)^{\alpha_t-1} Z_{k,I}
$$
where $Z_{k,I}\in\BC_k$, $I=(\alpha_{k+1},\dots,\alpha_n)$, and
$\alpha_j\in\{1,2\}$.

\medskip



Just as in the case of $\BC_2$, there exist idempotent bases in
$\BC_n$, that will be organized at each ``nested'' level $\BC_k$
inside $\BC_n$ as follows.  Denote by
\begin{align*}
  \e_{kl} &:= \frac{1+\i_k\i_l}{2}\,,\qquad
  \bar\e_{kl} := \frac{1-\i_k\i_l}{2}\,.
\end{align*}
Consider the following sets:
\begin{align*}
  S_1 &:= \{\e_{n-1,n}, \bar\e_{n-1,n}\},\\
  S_2 &:= \{\e_{n-2,n-1}\cdot S_1, \bar\e_{n-2,n-1}\cdot S_1\},\\
  &\vdots\\
  S_{n-1} &:= \{\e_{12}\cdot S_{n-2}, \bar\e_{12}\cdot S_{n-2}\}.
\end{align*}
At each stage $k$, the set $S_k$ has $2^k$ idempotents.  It is
possible to immediately verify the following
\begin{proposition}
  In each set $S_k$, the product of any two idempotents is zero.
\end{proposition}
We have several idempotent representations of $Z_n\in\BC_n$, as follows.
\begin{theorem}
  Any $Z_n\in\BC_n$ can be written as:
  $$
  Z_n=\sum_{j=1}^{2^k} Z_{n-k,j}\e_j,
  $$
  where $Z_{n-k,j}\in\BC_{n-k}$ and $\e_j\in S_k$.
\end{theorem}

Due to the fact that the product of two idempotents is $0$ at each
level $S_k,$ we will have many zero divisors in $\BC_n$ organized in
``singular cones''.

\medskip

In particular, at the last stage, we obtain:
\begin{proposition}
  Any $Z_n\in\BC_n$ admits the idempotent writing:
  \begin{align}
    \label{last_stage}
    Z_n=\sum_{j=1}^{2^{n-1}} \beta_j\e_j,
  \end{align}
  where $\beta_j$ are complex numbers (with respect to one fixed imaginary unit, say
  $\i$) and $\e_j\in S_{n-1}$. For a chosen imaginary unit $\i$, this decomposition is unique.
\end{proposition}
As in the case of bicomplex numbers, in this representation, the
multiplication of multicomplex numbers is component-wise, and yields
the isomorphism:
\begin{align*}
  \BC_n \simeq \sum_{j=1}^{2^{n-1}} \C(\i)\e_j\,.
\end{align*}

\medskip

This decomposition allows us to introduce a formula analogous
to~\eqref{shakira}. Explicitly, for $n=3$ in $\BC_3$, we have:
\begin{align}
  \label{shakiraDD}
  Z_3 &= Z_{21} + \i_3 Z_{22} = (Z_{21}-\i_2 Z_{22})\e_{23} + (Z_{21}+\i_2 Z_{22})\bar{\e}_{23}\nonumber\\
  &= \left((x_1+x_7+x_4-x_6)+ \i_1(x_2+x_8-x_3+x_5)\right)\e_1 \nonumber\\
  &+ \left((x_1+x_7-x_4+x_6)+ \i_1(x_2+x_8+x_3-x_5)\right)\e_2 \nonumber\\
  &+ \left((x_1-x_7+x_4+x_6)+ \i_1(x_2-x_8-x_3-x_5)\right)\e_3 \nonumber\\
  &+ \left((x_1-x_7-x_4-x_6)+ \i_1(x_2-x_8+x_3+x_5)\right)\e_4
\end{align}
This idempotent representation gives a component wise multiplicative
structure on $\C^4$:
\begin{align*}
  &\left((x_1 + \i_1 x_2) + \i_2(x_3 + \i_1 x_4)\right) +
  \i_3 \left((x_5 + \i_1 x_6) + \i_2(x_7 + \i_1 x_8)\right)\longmapsto\\
  &\left[(x_1+x_7+x_4-x_6)+ \i_1(x_2+x_8-x_3+x_5), (x_1+x_7-x_4+x_6)+ \i_1(x_2+x_8+x_3-x_5),\right.\\
  & \left.(x_1-x_7+x_4+x_6)+ \i_1(x_2-x_8-x_3-x_5),  (x_1-x_7-x_4-x_6)+ \i_1(x_2-x_8+x_3+x_5)\right]
\end{align*}

A tedious but straightforward computation shows that this is the only
multiplication on tricomplex numbers that is component-wise.

\bigskip

\section{Bicomplex and Multicomplex zeta functions}
\label{sec:bicomplex_zeta}

We consider first the case of bicomplex algebra $\BC=\BC_2$. 
Inside $\BC$ we consider the vector space $\BQ$ over $\Q$ 
\begin{equation}
  \label{eq:BQ}
  \BQ := \{ Z = x_1 +y_1\i + x_2\j + y_2\k\,\big|\, x_\ell,y_\ell\in\Q, \quad\ell=1,2\}\,.
\end{equation}

The vector space $\BQ$ can be equipped with a structure of
$\Q$-algebra, generated by the two variables $\i,\j$, with the
relations
\begin{align*}
  \i\j = \j\i,\qquad \i^2+1 = \j^2+1 =0\,.
\end{align*}
so that
\begin{align*}
  \BQ \simeq \Q[\i,\j]\big/\left(\i^2+1,\j^2+1\right)
\end{align*}
which is a commutative algebra.  Furthermore, note that
\begin{align*}
  \Q[X,Y]/(X^2+1,Y^2-1) \simeq \Q[X,Y]/(X^2+1,Y^2+1)
\end{align*}
by $X\mapsto X$ and $Y\mapsto XY$. Indeed,
\begin{align*}
  (XY)^2 + 1 &= X^2(Y^2-1) + (X^2+1)\in (X^2+1,Y^2-1),\\
  (XY)^2 - 1 &= X^2(Y^2+1) - (X^2+1)\in (X^2+1,Y^2+1)\,.
\end{align*}
More generally, if we let $K$ be a field and consider the $K$-algebra
\[
A= K[X,Y,Z]/(X^2+1, Y^2+1, Z-XY).
\]
We have
\[
A= K[X,Y]/(X^2+1, Y^2+1),
\]
because we can eliminate the variable $Z$ by using the relation
$Z= XY$. We can further write
\[
A= \left(K[X]/(X^2+1)\right)[Y]/(Y^2+1),
\]
and we denote $C:= K[X]/(X^2+1)$, and by $\i$ the class of
$X \in K[X]$ in this quotient. Then we get:
\[
A= C[Y]/(Y^2+1)= C[Y]/(Y+\i)(Y-\i)\simeq C[Y]/(Y+\i) \times C[Y]/(Y-\i),
\]
where the isomorphism (of $C$-algebras and hence of $K$-algebras) is
induced by the canonical map
\begin{align}
  \label{crucial}
  (\pi_1, \pi_2): C[Y]\to  C[Y]/(Y+\i) \times C[Y]/(Y-\i)
\end{align}
given by the two surjective maps $\pi_1, \pi_2$.  If we return to
$K[X,Y,Z]$, the images of $ X,Y$ and $Z$ in $A$ by the isomorphism
just considered are $ (\i,\i) $, $(-\i,\i)$ (note that the class of
$Y$ in the first factor is $-\i$, due to the relation $Y+\i= 0$) and
$(\i,\i)(-\i,\i)= (1,-1) $.  In conclusion, for
$x_1,x_2,x_3,x_4 \in K$, the class of $x_1+x_2X+x_3Y+x_4Z$ in $A$ is
\[
x_1(1,1)+x_2(\i,\i)+x_3(-\i,\i)+x_4(1,-1) = \left((x_1+x_4)
  +\i(x_2-x_3), (x_1-x_4) +\i(x_2+x_3) \right)
\]
which is exactly the map~\eqref{shakira} corresponding to the
idempotent representation of bicomplex numbers. 

\noindent Therefore, we obtain the following characterization of the
algebra $\BQ$:
\begin{theorem}
  \label{BQ:product}
  The algebra $\BQ$ is isomorphic to the product $\Q(\i)\times\Q(\i)$.
  The isomorphism is given by~\eqref{crucial} for $C=\Q$.
\end{theorem}
\begin{proof}
  Explicitly, we use the idempotent representation of bicomplex
  numbers in order to get a component-wise product of ideals, which is
  necessary for the Chinese Remainder theorem:
  \begin{align*}
    \BQ \simeq \Q(\i)\e + \Q(\i)\edag\,,
  \end{align*}
  and recall that the ideals $I_\e=\Q(\i)\e$ and $I_\edag=\Q(\i)\edag$
  are coprime in $\BQ$, with $I_\e\cdot I_\edag = \{0\}$. The Chinese
  Remainder theorem yields
  \begin{align*}
    \BQ  \simeq \BQ/I_\e \times \BQ/I_\edag = I_\edag \times I_\e
    \simeq \Q(\i)\times \Q(\i)\,,
  \end{align*}
  which is what we had to prove.
\end{proof}

\bigskip

The main objective now is to define a Dedekind-like zeta function for
an algebra which is the product of fields. According to
Artin~\cite{artin} and Hey~\cite{hey}, one can define a Dedekind-like
zeta function for hypercomplex algebras (such as the bicomplex and
multicomplex ones), if one considers ideals of maximal order in the
defining formula~\eqref{def:dedekind}. Moreover, it follows also that
the resulting zeta function for a product algebra will be the product
of the corresponding zeta functions of the factors, whenever the
multiplication is defined component-wise.

Therefore, the plan is to find the maximal order of the algebra $\BQ$
(see Theorem~\ref{thm:BQ} below), and to derive the ideal structure of
$\BQ$ (Lemma~\ref{product structure}), concluding with the main result
(Theorem~\ref{main}) defining the bicomplex zeta function.

We recall the following notions (see e.g.~\cite{reiner}). Let $R$ be a
commutative domain with quotient field $K$, and let $A$ be a
finite-dimensional $K$-algebra.  If a full $R$--lattice $\calL$ in $A$
(i.e. $\calL$ is a finitely generated $R$--submodule such that
$K\calL =A$) is a subring of $A$, then one says that $\calL$ is an
{\em $R$--order} in $A$. If, moreover, $\calL$ is not properly
contained in any $R$--order of $A$, then it is called a {\em maximal
  order}.

An algebra $A$ is called {\em semisimple} if it is isomorphic to a
direct sum of simple algebras (i.e. do not have non-trivial
subalgebras).  Furthermore, $A$ is called a {\em separable}
$K$--algebra if $A$ is semisimple and the center of each simple
component of $A$ is a separable field extension of $K$.

We now prove the following:
\begin{theorem}
  \label{thm:BQ}
  $\BQ$ is a semisimple separable algebra. Moreover, the maximal order
  of $\BQ$ is the product $\Z[\i]\times\Z[\i]$.
\end{theorem}
\begin{proof}
  The first statement follows from Theorem~\ref{BQ:product} above and
  a theorem of Weierstrass and Dedekind (see e.g.~\cite[Theorem 2.4.1
  (pp. 38)]{drozd}), which states that a commutative semisimple
  algebra is isomorphic to a direct product of fields, and,
  conversely, a direct product of fields is a semisimple algebra.

  Next, we recall~\cite[Theorem 10.5]{reiner} that in a separable
  algebra with a central idempotent decomposition
  \begin{align*}
    A = A_1 \oplus \dots \oplus A_t
  \end{align*}
  (such as the bicomplex and multicomplex algebras) every maximal
  order has a corresponding maximal order decomposition at each
  level. In particular, if one defines by $\{e_i\}$ the central
  idempotents of $A$ such that $e_ie_j=0$, $i\neq j$,
  $1= e_1 + \dots + e_t$, and $A_i=Ae_i$, then each maximal order is a
  direct sum of the maximal orders of each component.

  From Theorem~\ref{BQ:product} it follows that the maximal order of
  $\BQ$ is the product of the maximal orders of each factor, and the
  maximal order of $\Q(\i)$ is $\Z[\i]$.
\end{proof}

\bigskip

In order to derive the ideal structure of the algebra $\BQ$, we recall
a more general classical lemma on the ideal structure of a product of
two unitary commutative rings.
\begin{lemma}
  \label{product structure}
  Let $R=R_1\times R_2$ be the product of two unitary commutative
  rings. If $\calI(A)$ is the monoid of ideals of a ring $A$, with the
  binary operation $(a,b)\mapsto a\cdot b$, then
  \begin{align*}
    \calI(R) = \calI(R_1)\times\calI(R_2)\,.
  \end{align*}
\end{lemma}
\begin{proof}
  To begin with, if $I_1$ is an ideal of $R_1$ then $I_1\times\{0\}$
  is an ideal of the product. Moreover
  \begin{align*}
    R/I_1 \simeq R_1/I_1 \times R_2\,.
  \end{align*}
  Similarly, for an ideal $I_2$ of $R_2$. Moreover, if $I_1$ is an ideal
  of $R_1$ and $I_2$ is an ideal of $R_2$, then $I_1\times I_2$ is an
  ideal of $R_1\times R_2$, and
  \begin{align*}
    (R_1\times R_2)/(I_1\times I_2) \simeq R_1/I_1 \times R_2/I_2\,.
  \end{align*}
  Reciprocally, if $R_1$ and $R_2$ are two commutative rings and $I$ is
  an ideal of $R_1\times R_2$, we define:
  \begin{align*}
    I_1 &:= \{ r_1\in R_1\,\big|\, \exists r_2\in R_2, (r_1,r_2)\in I\}\\
    I_2 &:= \{ r_2\in R_2\,\big|\, \exists r_1\in R_1, (r_1,r_2)\in I\}\,,
  \end{align*}
  then $I_1, I_2$ are ideals of $R_1, R_2$, respectively, and
  $I=I_1\times I_2$. Indeed, observe first that $I_1\times \{0\}$ and
  $\{0\}\times I_2$ are ideals of $R$, contained in $I$: if
  $i_1\in I_1$, then for some $r_2\in R_2$, $(i_1,r_2)\in I$, and then
  \begin{align*}
    (i_1,0)=(i_1,r_2) (1,0)\,,
  \end{align*}
  and similarly for $I_2$. Therefore
  \begin{align*}
    I_1\times I_2 = (I_1\times \{0\}) + (\{0\}\times I_2)\subset I\,.
  \end{align*}
  Conversely, if $(x,y)\in I$, then
  \begin{align*}
    (x,y) = (x,0)(1,0) + (0,y)(0,1)\in I_1\times I_2\,.
  \end{align*}
\end{proof}

\bigskip

\noindent It follows that ideals of maximal order in $\BQ$ are
products of ideals of maximal order in each one of the factors
$\Q(\i)$, i.e. products of ideals of $\Z[\i]$. 

We have everything necessary now to prove our main result:
\begin{theorem}
  \label{main}
  The Dedekind-like (Hey) zeta function of the algebra $\BQ$ is
  \begin{align*}
    \zeta_{\BQ}(s) = \zeta(s)^2\cdot L(\chi_{-4},s)^2\,.
  \end{align*}
  Moreover, $\zeta_{\BQ}$ has a double pole at $s=1$, a residue
  equal to $\displaystyle\frac{\pi}{2}\gamma_{\Q(\i)}$, and it
  verifies the functional equation:
  \begin{align*}
    \pi^{-(2-2s)} \Gamma^2(1-s)\zeta_{\BQ}(1-s) =
    \pi^{-2s} \Gamma^2(s)\zeta_{\BQ}(s)\,.
  \end{align*}
\end{theorem}
\begin{proof}
  The Hey zeta function of the algebra $\BQ$ is defined as
  in~\eqref{def:dedekind} for all ideals of maximal order of
  $\BQ$. According to Theorems~\ref{BQ:product} $\BQ$ is isomorphic to
  the product $\Q(\i)\times \Q(\i)$, and using Theorem~\ref{thm:BQ}
  and Lemma~\ref{product structure}, the maximal order ideals of $\BQ$
  are products of maximal ideals of $\Z[\i]$. It follows that the zeta
  function of $\BQ$ is the product of the two respective Dedekind zeta
  functions of $\Q(\i)$.  Since we know the
  expression~\eqref{eq:dedekind_zeta} for the Dedekind zeta function
  for the field $\Q(\i)$, this proves the first part of the theorem.

  As for the second part, it follows from the functional equations of
  the Dirichlet beta and of the Riemann zeta functions written in the
  following form:
  \begin{align*}
    \beta(1-s) &= \left(\frac\pi2\right)^{-s} \sin\left(\frac\pi2 s\right)
                 \Gamma(s)\beta(s)\,,\\
    \zeta(1-s) &= \frac{1}{2^s \pi^{s-1} \sin\left(\frac\pi2 s\right)\Gamma(1-s)}\zeta(s)\,.
  \end{align*}
  This concludes the proof of our main result.
\end{proof}

\bigskip

The Dirichlet $L$-series $L(\chi_{-4},s)$ is analytic in
$\{\Re(s)>0\}$ by general principles on Dirichlet series. This follows
from the fact that if $(a_n)$ is a sequence of complex numbers such
that there exist $C>0$ and $r>0$ such that for large $n$, we have:
\begin{align*}
  \left|\sum_{k=1}^n a_k\right|\leq C n^r\,,
\end{align*}
then the Dirichlet series
\begin{align*}
  f(s) = \sum_{n\geq 1} \frac{a_n}{n^s},\qquad s\in\C
\end{align*}
is analytic on $\{\Re(s)>0\}$.  
Now if $K$ is a quadratic field with discriminant $\Delta$,
the quadratic character $\chi_K$,
\begin{align*}
  \chi_K(m)=\left(\frac\Delta m\right)
\end{align*}
is periodic of period $|\Delta|$, so for any $n$,
\begin{align*}
  \sum_{n=n_0}^{n_0+|\Delta|-1} \chi_K(n) =0
\end{align*}
and there exists $C>0$ such that
\begin{align*}
  \left|\sum_{k=1}^n \chi_K(k)\right|\leq C.
\end{align*}
The analytic (and meromorphic) continuation of $\beta$ (and $\zeta$)
can also be deduced from the integral formulas:
\begin{align*}
  \beta(s) &= \frac{1}{\Gamma(s)} \int_0^\infty \frac{t^{s-1}}{e^t + e^{-t}}\, dt,\qquad \Re(s)>0,\\
  \zeta(s) &= \frac{1}{(1-2^{1-s})\Gamma(s)} \int_0^\infty \frac{t^{s-1}}{e^t + 1}\, dt.
\end{align*}

From the decomposition in infinite products of $L(\chi_{-4},s)$ and of
the Riemann zeta function, we obtain:
\begin{align}
  \zeta_{\BQ}(s) &= \left(\prod_{p\equiv 1\mod 4}\left(1-p^{-s}\right)^{-2}
                   \prod_{p\equiv 3\mod 4}\left(1-p^{-2s}\right)^{-2}\left(1-2^{-s}\right)^{-1}\right)^2\nonumber\\
                 &= \left( \sum_{n\geq 0} \frac{(-1)^n}{(2n+1)^s}\right)^2\,.
\end{align}

\bigskip

\begin{remark}
  We extend Remark~\ref{remark} to the case of our bicomplex zeta
  function.  Let $a_n$ be the Dirichlet series coefficients of
  $\zeta_{\Q(\i)}(s)$ and let $A_n$ be the corresponding coefficients
  of $\zeta_{\BQ}(s)$. Then
  \begin{align*}
    A_n= \sum_{p,q\,;\,  pq=n}a_p a_q= \frac{1}{16} \sum_{p,q\,;\, pq=n}r_2(p) r_2(q)   \,.
  \end{align*}
  In particular $A_1= 1$ and if $n$ is prime
  \begin{align*}
    A_n= \frac{1}{2}r_2(n)  
  \end{align*}
  because $r_2(1)= 4$.  If moreover, $n$ is prime congruent to 1 (mod
  4), by Fermat theorem we get $A_n= 4$. So we have precise
  information about the Dirichlet coefficients $A_n$ of our bicomplex
  zeta function $\zeta_{\BQ}(s)$, when $n$ is prime and a fixed value
  when $n$ is prime congruent to 1(mod 4).  The first coefficients of
  the Dirichlet series of the bicomplex zeta series are:
  $$
  A_1= 1, A_2=2, A_3= 0, A_4= 3, A_5= 4,\dots
  $$
\end{remark}

\bigskip

We now conclude with the general case of the multicomplex algebra
$\BC_n$. The definitions and proofs follow closely the case $n=2$ of
bicomplex numbers, so, for simplicity, we just state the main result.

The corresponding ``rational'' subalgebras $\BQ_n$ are defined
analogously to~\eqref{eq:BQ}.  The component-wise multiplication given
by the idempotent representation in the last stage
(see~\eqref{last_stage} and~\eqref{shakiraDD}) produces the splitting
of $\BQ_n$ into $2^{n-1}$ factors of $\Q(\i_1)$:
\begin{align*}
  \BQ_n  \simeq \prod_{\ell=1}^{2^{n-1}} \BQ_n/I_{\e_\ell}
  \simeq \Q(\i_1)^{2^{n-1}}\,.
\end{align*}
As before, the maximal order of the product above is the product of
the maximal orders of $\Q(\i_1)$ (which is $\Z[\i_1]$). In complete
analogy, we obtain the following expression for the associated zeta
function of multicomplex numbers:

\begin{theorem}
  The Dedekind-like zeta function of the algebra $\BQ_n$ is
  \begin{align}
    \zeta_{\BQ_n}(s) = \zeta(s)^{2^{n-1}}\cdot L(\chi_{-4},s)^{2^{n-1}}\,.
  \end{align}
\end{theorem}

The multicomplex zeta functions above has a pole of order $2^{n-1}$ at
$s=1$ and a residue equal to
$\displaystyle\frac{\pi}{2}\gamma_{\Q(\i)}$. It verifies the
functional equation:
\begin{align}
  \pi^{-2^{n-1}(1-s)} \Gamma^{2^{n-1}}(1-s)\zeta_{\BQ_n}(1-s) =
  \pi^{-2^{n-1}s} \Gamma^{2^{n-1}}(s)\zeta_{\BQ_n}(s)\,.
\end{align}

\bigskip


\bibliographystyle{unsrt}
\bibliography{bibliography}

\end{document}